\documentclass[10pt]{amsart}
\usepackage[T1]{fontenc}
\usepackage{lmodern}
\usepackage{amsmath,amssymb,amsthm,mathtools}
\usepackage{microtype}
\usepackage{needspace}
\usepackage{enumitem}
\usepackage[hidelinks]{hyperref}
\hypersetup{pdftitle={Nonexistence of maximal curves of genus five over F64},
  pdfauthor={Gilberto B. Almeida Filho, Saeed Tafazolian, Stéfani C. Vieira}}

\newtheorem{theorem}{Theorem}[section]
\newtheorem{proposition}[theorem]{Proposition}
\newtheorem{lemma}[theorem]{Lemma}
\newtheorem{corollary}[theorem]{Corollary}
\theoremstyle{remark}
\newtheorem{remark}[theorem]{Remark}
\numberwithin{equation}{section}

\newcommand{\F}{\mathbb F}
\newcommand{\PP}{\mathbb P}
\newcommand{\OO}{\mathcal O}
\newcommand{\Ca}{\mathcal C}
\newcommand{\im}{\operatorname{im}}
\newcommand{\rk}{\operatorname{rank}}
\newcommand{\divi}{\operatorname{div}}
\newcommand{\adj}{\operatorname{adj}}
\newcommand{\diag}{\operatorname{diag}}
\newcommand{\GL}{\operatorname{GL}}
\newcommand{\one}{\mathbf 1}
\newcommand{\xx}{\mathbf X}

\title[Maximal curves of genus five over $\F_{64}$]
{Nonexistence of maximal curves of genus five over $\F_{64}$}
\author{Gilberto B. Almeida Filho}
\address{Departamento de Ciências Exatas e da Terra, Universidade Federal de Mato Grosso, Várzea Grande, Mato Grosso, Brazil}
\email{gilberto.filho@ufmt.br}
\thanks{All three authors contributed equally to this work.}

\author{Saeed Tafazolian}
\address{IMECC, Universidade Estadual de Campinas, Campinas, São Paulo, Brazil}
\email{saeed@unicamp.br}

\author{Stéfani C. Vieira}
\address{Departamento de Ciências Exatas e da Terra, Universidade Federal de Mato Grosso, Várzea Grande, Mato Grosso, Brazil}
\email{stefani.vieira1@ufmt.br}
\date{}

\begin{document}
\begin{abstract}
We show that there is no maximal curve of genus five over $\F_{64}$.
As a consequence, $N_{64}(5)=140$, and the genus spectrum of maximal
curves over $\F_{64}$ is determined. The proof uses the vanishing of
the third iterate of the Cartier operator. We prove that a
nonhyperelliptic curve of genus five in characteristic two satisfying
this condition is nontrigonal. Its canonical theta characteristic
defines a separable cover of degree four with one geometric branch
value. The two possible ramification types give either a rational
subcanonical point or a point bound obtained from the cubic resolvent.
Both cases exclude maximality over $\F_{64}$.
\end{abstract}
\maketitle

\section{Introduction}\label{sec:intro}

Let $C$ be a smooth, projective, geometrically irreducible curve of
genus $g$ defined over a finite field $\F_Q$. The Hasse--Weil bound is
$\#C(\F_Q)\le Q+1+2g\sqrt Q$. When $Q=q^2$, the curve is called maximal if
\[
 \#C(\F_{q^2})=q^2+1+2gq.
\]
We denote by $N_Q(g)$ the largest number of rational points on a curve
of genus $g$ over $\F_Q$.

In this paper we consider maximal curves of genus five over $\F_{64}$.
In this case the Hasse--Weil bound is $145$. The manYPoints entry for
$Q=64$, $g=5$, submitted by E.~W. Howe, gives a curve with $140$ rational
points and records the exclusion of $141,142,143,144$ by the methods of
Howe--Lauter~\cite{HL}. We give a direct verification of these numerical
inputs in Appendix~\ref{app:numerical}. Thus
\begin{equation}\label{eq:known-values}
 N_{64}(5)\in\{140,145\}.
\end{equation}
The existence of a maximal curve of genus five is also the remaining
question in the determination of the genus spectrum over $\F_{64}$;
see~\cite[Section~4]{ATT}. Our main result answers this question.

\begin{theorem}\label{thm:main}
There is no maximal curve of genus five over $\F_{64}$.
Consequently, $N_{64}(5)=140$.
\end{theorem}

The main tool is the Cartier operator $\Ca$. A maximal curve over
$\F_{2^{2m}}$ satisfies $\Ca^m=0$; see~\cite[Theorem~3.3]{GT}.
For a nonhyperelliptic curve of genus five, the condition $\Ca^3=0$
and Clifford's theorem give $a(C)=2$. The canonical theta
characteristic thus has two independent sections.
We use the associated pencil to prove the following result over any
finite field of characteristic two.

\begin{theorem}\label{thm:structural}
Let $C/\F_Q$ be a nonhyperelliptic curve of genus five in characteristic
two, and suppose that $\Ca^3=0$. Then $C$ is nontrigonal, and its
canonical theta characteristic $\theta$ defines a separable morphism
\[
 \rho:C\longrightarrow\PP^1
\]
of degree four with a unique geometric branch value
$b\in\PP^1(\F_Q)$. Exactly one of the following holds:
\begin{enumerate}[label=\textup{(\roman*)},leftmargin=2.2em]
\item $\rho^{-1}(b)=4P$ for a point $P\in C(\F_Q)$, and $K_C\sim8P$.
\item $\rho^{-1}(b)=2P_1+2P_2$ for distinct geometric points $P_1,P_2$.
      The cubic resolvent of $\rho$ is geometrically irreducible of
      genus zero, and
      \[
       \#C(\F_Q)\le2Q+1.
      \]
\end{enumerate}
\end{theorem}

To prove this theorem, we first exclude the trigonal case by applying
the plane Cartier formula of St\"ohr--Voloch~\cite{SV}.
In the nontrigonal case, the canonical model is an intersection of
three quadrics. The coefficient formula for the Hasse--Witt matrix
gives a normal form in which the relative Jacobian determinant of
$\rho$ is the fourth power of a linear form. The two possible
exceptional fibers are then treated separately. For the fiber
$2P_1+2P_2$, we give an explicit parametrization of the cubic resolvent
and deduce the point bound.

For a maximal curve of genus five over $\F_{64}$, the first alternative
in Theorem~\ref{thm:structural} is excluded by the fundamental
equivalence for maximal curves. The second gives $145\le129$.
This proves Theorem~\ref{thm:main}.

In Section~\ref{sec:prelim} we collect the facts used in the proofs.
Section~\ref{sec:trigonal} treats the trigonal case.
The normal form and the cubic resolvent are studied in
Sections~\ref{sec:normal} and~\ref{sec:resolvent}.
The application to $\F_{64}$ is given in Section~\ref{sec:application}.
The appendices contain the full coefficient matrix and the numerical
verification used in~\eqref{eq:known-values}.

\section{Preliminaries}\label{sec:prelim}

Throughout, a curve is smooth, projective and geometrically
irreducible. Hyperellipticity and trigonality are understood over an
algebraic closure. Unless otherwise stated, $k$ is algebraically closed
of characteristic two. We write $K_C$ for a canonical divisor and
$J_C$ for the Jacobian, and use additive notation for line bundles.

\subsection{Cartier and the canonical theta characteristic}

Let $\Ca$ denote the Cartier operator on rational differentials. For a
separating function $z\in k(C)$, every rational differential is written
uniquely as
\[
 (a^2+b^2z)\,dz,\qquad a,b\in k(C),
\]
and
\[
 \Ca\bigl((a^2+b^2z)\,dz\bigr)=b\,dz.
\]
Thus $\Ca$ is inverse-Frobenius-semilinear and its kernel consists of
exact differentials.

The divisor of $dz$ has even coefficients, since only odd powers
contribute to the derivative of a local Laurent expansion. Thus
\begin{equation}\label{eq:theta}
 B_z=\tfrac12\divi(dz),\qquad \theta=\OO_C(B_z)
\end{equation}
defines a theta characteristic of degree $g(C)-1$.
If $z'$ is another separating function, then $dz'/dz$ is a square in
$k(C)$, so the class of $\theta$ is independent of $z$.
There is a semilinear bijection
\begin{equation}\label{eq:theta-kernel}
 H^0(C,\theta)\longrightarrow
 \ker\bigl(\Ca:H^0(C,\omega_C)\to H^0(C,\omega_C)\bigr),
 \qquad h\longmapsto h^2dz.
\end{equation}
Thus
\[
 \theta^{\otimes2}\simeq\omega_C,\qquad
 h^0(\theta)=a(C):=\dim_k\ker\Ca.
\]
We call $\theta$ the canonical theta characteristic; see
\cite[Propositions~3.1 and~3.3]{SV} and \cite[Remark~4.3]{Dra26}.
The construction works over a perfect field. In particular, if $C$
is defined over $\F_Q$, then $\theta$ and its complete linear system
are defined over $\F_Q$.

\begin{lemma}\label{lem:rank}
Let $C$ be nonhyperelliptic of genus five and suppose that $\Ca^3=0$.
Then
\begin{equation}\label{eq:ranks}
 h^0(\theta)=2,\qquad \rk\Ca=3,\qquad \rk\Ca^2=1.
\end{equation}
\end{lemma}
\begin{proof}
Since $\Ca^3=0$, the kernel filtration gives
$5\le3\dim\ker\Ca$, and hence $\dim\ker\Ca\ge2$.
Clifford's theorem gives $h^0(\theta)\le3$. Equality would imply that
$C$ is hyperelliptic, as $\theta$ has degree four and is neither
trivial nor canonical. Thus $h^0(\theta)=2$ and $\rk\Ca=3$.
The restriction of $\Ca$ to its image has square zero. Since that
image has dimension three, $\rk\Ca^2\le1$. Equality holds, for
otherwise $\im\Ca\subseteq\ker\Ca$, which is impossible.
\end{proof}

Let $F$ denote Frobenius on $H^1(C,\OO_C)$. Serre duality satisfies
\[
 \langle Fv,\omega\rangle=\langle v,\Ca\omega\rangle^2.
\]
For dual bases, if $H$ is the matrix of $F$, then
\[
 F(v)=Hv^{(2)},\qquad
 F^3(v)=HH^{(2)}H^{(4)}v^{(8)}.
\]
Superscripts in parentheses denote entrywise powers.
The Cartier matrix in the dual basis is $(H^t)^{(1/2)}$.
Thus a canonical coordinate in $\ker\Ca$ corresponds to a zero row
of $H$, and $\Ca^3=0$ is equivalent to $F^3=0$.
All statements about nilpotence below concern semilinear operators.
For a Frobenius-semilinear operator with matrix $A$, the square-zero
condition is $AA^{(2)}=0$.

\subsection{Two formulas for the Cartier and Hasse--Witt matrices}

A nonhyperelliptic, nontrigonal curve of genus five has a canonical
model
\[
 C=V(q_1,q_2,q_3)\subset\PP^4
\]
which is a complete intersection of three quadrics; see, for example,
\cite{Dra24,Dra26}. Put $\xx=(X_0,\ldots,X_4)$ and let $e_i$ be the
$i$th standard vector. With compatible residue and Serre-dual bases,
its Hasse--Witt matrix is
\begin{equation}\label{eq:HW-formula}
 H_{ij}=[\xx^{\one+2e_j-e_i}](q_1q_2q_3),
 \qquad 0\le i,j\le4.
\end{equation}
This is \cite[Equation~(2.2)]{Dra26}. It also follows directly from the
Koszul description
\[
 H^1(C,\OO_C)\simeq H^4(\PP^4,\OO_{\PP^4}(-6)).
\]
The latter space has basis represented by
$1/(X_0\cdots X_4X_i)$; Frobenius squares a representative and then
multiplies by $q_1q_2q_3$, giving~\eqref{eq:HW-formula}.
At fixed coordinates, replacing the generators by
$(q'_1,q'_2,q'_3)^t=R(q_1,q_2,q_3)^t$ multiplies $H$ by $\det R$.
Indeed, a term with a repeated factor has the form $q_i^2q_j$ and
cannot contribute to a monomial with four odd exponents, whereas
every exponent vector in~\eqref{eq:HW-formula} has four odd entries.
Coordinate changes also change the cohomology basis; rank and
semilinear nilpotence are intrinsic.

For a plane model $\Phi(x,y)=0$ with $\Phi_y\ne0$, set
$\omega=dx/\Phi_y$. The plane Cartier formula in characteristic two is
\begin{equation}\label{eq:plane-Cartier}
 \Ca(h\omega)=
 \left(\frac{\partial^2(\Phi h)}{\partial x\,\partial y}\right)^{1/2}
 \omega.
\end{equation}
For polynomial $h$, the expression under the square root has only
even powers of $x,y$. We apply this formula to adjoint polynomials,
which represent regular differentials on the normalization;
see~\cite{SV}.

\subsection{Maximal curves}

For a maximal curve over $\F_{q^2}$, the Frobenius endomorphism of its
Jacobian is $[-q]$. If $P_0\in C(\F_{q^2})$ and $\Phi_{q^2}$ is the
Frobenius morphism of the curve, the fundamental equivalence is
\begin{equation}\label{eq:fundamental}
 qP+\Phi_{q^2}(P)\sim(q+1)P_0
\end{equation}
for every geometric point $P$; see~\cite[Section~1]{KT}.
In particular, $(q+1)P\sim(q+1)P_0$ for rational $P$.
The following lemma combines~\cite[Theorem~3.3]{GT} with a consequence
of~\eqref{eq:fundamental}.

\begin{lemma}\label{lem:maximal}
A maximal curve over $\F_{2^{2m}}$ satisfies $\Ca^m=0$.
If $C/\F_{64}$ is maximal of genus five, there is no
$P\in C(\F_{64})$ with $K_C\sim8P$.
\end{lemma}
\begin{proof}
The first assertion is~\cite[Theorem~3.3]{GT}, with $p=2$ and $n=m$.

Now let $g(C)=5$ and suppose that $K_C\sim8P$ for a rational point $P$.
Riemann--Roch gives
\[
 h^0(8P)=5=h^0(9P),
\]
so $P$ is a base point of $|9P|$. Choose a rational point $P_0\ne P$;
such a point exists since $\#C(\F_{64})=145$.
By~\eqref{eq:fundamental}, the divisor $9P_0$ belongs to $|9P|$ and
does not contain $P$. This is a contradiction.
\end{proof}

\section{The trigonal case}\label{sec:trigonal}

We first exclude the trigonal case.

\begin{proposition}\label{prop:trigonal}
A nonhyperelliptic trigonal curve of genus five in characteristic two
does not satisfy $\Ca^3=0$.
\end{proposition}
\begin{proof}
Suppose that $C$ is trigonal and $\Ca^3=0$, and let $A$ be a
trigonal pencil. By Lemma~\ref{lem:rank}, $h^0(\theta)=2$.
Riemann--Roch gives
\[
 \deg(K_C-A)=5,\qquad h^0(K_C-A)=3.
\]
The series $|K_C-A|$ is base-point-free. Otherwise a line bundle of
degree four would have at least three sections, contrary to the
equality case of Clifford's theorem for a nonhyperelliptic curve. The resulting morphism to $\PP^2$ has nondegenerate image.
Its degree times the degree of its image is five, so it is birational
onto a plane quintic. The quintic has arithmetic genus six and
normalization of genus five. Hence it has one singular point of
$\delta$-invariant one: an ordinary node or an ordinary cusp;
see~\cite[Sections~2--3]{Wennink}.

Put the singular point at $(0:0:1)$ and choose affine coordinates with
$\Phi_y\ne0$. Adjoint conics give the canonical basis
\begin{equation}\label{eq:adjoint-basis}
 v_0=x^2\omega,\quad v_1=xy\omega,\quad v_2=y^2\omega,
 \quad v_3=x\omega,\quad v_4=y\omega,
 \qquad \omega=\frac{dx}{\Phi_y}.
\end{equation}
The net of canonical quadrics is spanned by
\begin{equation}\label{eq:trig-net}
 v_0v_2+v_1^2,\qquad
 v_0v_4+v_1v_3,\qquad
 v_1v_4+v_2v_3.
\end{equation}
Indeed, the products of the five homogeneous adjoint conics span
twelve quartic monomials. A quartic vanishing on the irreducible
quintic must be zero. Thus the three displayed relations form the
kernel of the multiplication map on quadratic expressions in the
canonical coordinates.

Here the rank of a quadric is the least number of linear variables
needed to express it, rather than the rank of its polar form.
For the linear combination of~\eqref{eq:trig-net} with coefficients
$(\lambda,\mu,\nu)$, the five principal four-by-four Pfaffians of its
polar matrix are
\[
 \nu^2,\quad\mu\nu,\quad\mu^2,\quad\lambda\nu,\quad\lambda\mu.
\]
Thus the only rank-three quadric in the net is
$v_0v_2+v_1^2$.
If $s_0,s_1$ is a basis of $H^0(\theta)$, the three independent
canonical sections $s_0^2,s_0s_1,s_1^2$ also give a rank-three quadric.
The image of its polar form, viewed as a map
$H^0(C,\omega_C)^*\to H^0(C,\omega_C)$, is the span of $s_0^2,s_1^2$.
By~\eqref{eq:theta-kernel} and uniqueness of the rank-three quadric,
\begin{equation}\label{eq:trig-kernel}
 \ker\Ca=\langle x^2\omega,y^2\omega\rangle.
\end{equation}
In particular, $\Ca(\omega)=0$, since
$\Ca(x^2\omega)=x\Ca(\omega)$.

If the double point is a node, choose coordinates with quadratic part
$xy$. Then
\[
 \frac{\partial^2\Phi}{\partial x\,\partial y}
 =1+\alpha x^2+\beta y^2
\]
for some $\alpha,\beta\in k$.
Its square root is a nonzero linear polynomial, and cannot vanish in
the function field of the quintic. Formula~\eqref{eq:plane-Cartier}
contradicts $\Ca(\omega)=0$.

The singular point is therefore a cusp. Choose coordinates such that
its tangent is $y=0$, the coefficient of $y^2$ is one, and
$\Phi_y\ne0$. These conditions are compatible: if necessary, replace
$x$ by $x+cy$ for a suitable $c\in k$. Write
\begin{align}\label{eq:cusp}
 \Phi={}&y^2+a x^3+b x^2y+cxy^2+d y^3\notag\\
 &+e x^4+f x^3y+g x^2y^2+hxy^3+i y^4\notag\\
 &+j x^5+\kappa x^4y+\ell x^3y^2+m x^2y^3+nxy^4+o y^5,
 \qquad a\ne0.
\end{align}
Here $a\ne0$ follows from the fact that the cusp has
$\delta$-invariant one: after blowing up in the chart $y=xy_1$,
the strict transform starts with $y_1^2+ax$ and must be smooth.
Since
\[
 \Ca(\omega)=(\sqrt f\,x+\sqrt h\,y)\omega,
\]
we have $f=h=0$.
On the quotient of $H^0(C,\omega_C)$ by~\eqref{eq:trig-kernel}, with
basis $xy\omega,x\omega,y\omega$, Cartier is represented by
\begin{equation}\label{eq:cusp-quotient}
 B=\begin{pmatrix}
 \sqrt g&\sqrt m&\sqrt\ell\\
 0&\sqrt b&\sqrt a\\
 1&\sqrt d&\sqrt c
 \end{pmatrix}.
\end{equation}
The first and third columns are independent because $a\ne0$, so
$\rk B\ge2$.
On the other hand, $\Ca^3=0$ implies
$\Ca^2H^0(C,\omega_C)\subseteq\ker\Ca$. The induced operator on the
three-dimensional quotient has square zero, so its rank is at most
one. This is a contradiction.
\end{proof}

\section{A normal form for the half-canonical cover}\label{sec:normal}

Let $C/k$ be nonhyperelliptic of genus five with $\Ca^3=0$.
By Proposition~\ref{prop:trigonal}, it is nontrigonal, and
Lemma~\ref{lem:rank} gives $h^0(\theta)=2$. The pencil $|\theta|$ is
base-point-free. Indeed, a base divisor would leave a pencil of degree
at most three, contrary to the gonality of $C$.

Choose a basis $s_0,s_1$ of $H^0(C,\theta)$ and put
\[
 X=s_0s_1,\qquad Y=s_0^2,\qquad Z=s_1^2.
\]
Complete these to canonical coordinates $X,Y,Z,T,U$.
Then $Y,Z$ span $\ker\Ca$ and
\[
 q_1=X^2+YZ
\]
vanishes on the canonical curve. Its vertex $X=Y=Z=0$ is disjoint from
$C$.

\begin{lemma}\label{lem:normal}
In suitable canonical coordinates, the equations of $C$ are
\begin{align}\label{eq:normal}
 q_1={}&X^2+YZ,\notag\\
 q_2={}&T^2+Q_2(X,Y,Z)+(gY+hZ)T+(jY+\kappa Z)U,\\
 q_3={}&U^2+Q_3(X,Y,Z)+(sY+\tau Z)T+(vY+wZ)U,\notag
\end{align}
where $Q_2,Q_3$ are ternary quadratic forms and
\begin{equation}\label{eq:cross}
 gw+hv+j\tau+\kappa s=0,
 \qquad (gv+js,\,hw+\kappa\tau)\ne(0,0).
\end{equation}
\end{lemma}
\begin{proof}
On the vertex line of $q_1$, the remaining canonical quadrics restrict
to independent binary quadrics with no common zero. Their span
contains a nonzero square, since the coefficient of $TU$ is a linear
functional. After changing $T,U$ and the equations, we may take the
first restriction to be $T^2$. The coefficient of $U^2$ in the second
is then nonzero. Rescaling and subtracting a multiple of the first
restriction gives $U^2+\ell TU$. Using $q_1$ to remove $X^2$-terms,
we obtain
\begin{align}\label{eq:general}
 q_2={}&T^2+aXY+bXZ+cY^2+dYZ+eZ^2\notag\\
      &+(fX+gY+hZ)T+(iX+jY+\kappa Z)U,\\
 q_3={}&U^2+\ell TU+mXY+nXZ+oY^2+pYZ+\zeta Z^2\notag\\
      &+(rX+sY+\tau Z)T+(uX+vY+wZ)U.\notag
\end{align}

Let $H$ be the Hasse--Witt matrix computed by~\eqref{eq:HW-formula};
its full expression is given in Appendix~\ref{app:HW}.
Its $Y$- and $Z$-rows vanish, since $Y,Z$ span $\ker\Ca$.
The changes of $T,U$ used below preserve these canonical sections.
This vanishing is equivalent to
\begin{align}
 b\ell+fw+hu+i\tau+\kappa r&=0,\label{eq:zero-row1}\\
 a\ell+fv+gu+is+jr&=0.\label{eq:zero-row2}
\end{align}
By Lemma~\ref{lem:rank}, $\rk H=3$, so the image of Frobenius is
\[
 S=\langle e_X,e_T,e_U\rangle.
\]
Write $D=H_{XX}$, $E=H_{TX}$ and $G=H_{UX}$.
The restriction of Frobenius to $S$ has matrix
\begin{equation}\label{eq:A}
 A=\begin{pmatrix}
 D&\ell&0\\
 E&f\ell+u&i\\
 G&r&f+i\ell
 \end{pmatrix}.
\end{equation}
Since $S=\im F$ and $F^3=0$, the restriction has square zero:
\begin{equation}\label{eq:A-square-zero}
 AA^{(2)}=0.
\end{equation}
As $\dim S=3$, it follows that $\rk A\le1$.

\smallskip\noindent\emph{First, $\ell=0$.}
Suppose that $\ell\ne0$. The second column of~\eqref{eq:A} has nonzero
first entry, while the third has zero first entry. Rank at most one
therefore forces the third column to vanish, giving $i=f=0$.
The substitutions
\[
 T\longmapsto T+(u/\ell)X,\qquad
 U\longmapsto U+(r/\ell)X
\]
set $u=r=0$ and preserve $f=i=0$; terms in $X^2$ are again removed
using $q_1$. In these coordinates, rank at most one gives $E=G=0$,
and the $(1,2)$-entry of~\eqref{eq:A-square-zero} gives
$D\ell^2=0$, hence $D=0$.
Equations~\eqref{eq:zero-row1}--\eqref{eq:zero-row2} now yield $a=b=0$.

The $T$- and $U$-rows of $H$ are then supported in the columns $Y,Z$ and
form the matrix
\[
 \begin{pmatrix}jm&\kappa n\\gm&hn\end{pmatrix}.
\]
The equalities $E=G=0$ are
\[
 jn+\kappa m=0,\qquad gn+hm=0.
\]
Multiplying these equalities by $h$ and $\kappa$, respectively,
gives $n(jh+\kappa g)=0$. Hence the determinant $mn(jh+\kappa g)$
vanishes, and the displayed matrix has rank at most one. The $X$-row adds at most one to the rank of $H$, contradicting
$\rk H=3$. Hence $\ell=0$.

\smallskip\noindent\emph{Next, $f=i=r=u=0$.}
The space $W=\langle e_T,e_U\rangle$ is now Frobenius-stable, with
restriction matrix
\begin{equation}\label{eq:B}
 B=\begin{pmatrix}u&i\\r&f\end{pmatrix}.
\end{equation}
It satisfies $BB^{(2)}=0$. To describe the allowed changes of
basis, put
\[
 M=\begin{pmatrix}f&i\\r&u\end{pmatrix},\qquad B=\adj(M).
\]
For $P\in\GL_2(k)$, substitute $(T,U)^t=P(T',U')^t$ in the two
equations and multiply them by $(P^{(2)})^{-1}$ to renormalize their
pure squares. The coefficient matrix of the terms multiplied by $X$
becomes
\[
 M'=(P^{(2)})^{-1}MP.
\]
Taking adjugates gives
\begin{equation}\label{eq:semilinear-change}
 B'=\adj(M')=(\det P)^{-1}P^{-1}BP^{(2)}.
\end{equation}
Thus the allowed coordinate changes induce semilinear changes of
basis on $W$, up to a nonzero scalar. The scalar does not affect
the image, the kernel or the square-zero condition.

If $B\ne0$, the corresponding operator has rank one, and its image
and kernel are the same line. By~\eqref{eq:semilinear-change}, we may
take this line to be $ke_U$ and write
\[
 B=\begin{pmatrix}0&0\\r&0\end{pmatrix},\qquad r\ne0.
\]
Thus $u=i=f=0$. Equations~\eqref{eq:zero-row1} and
\eqref{eq:zero-row2} give $\kappa=j=0$.
Since the second column of~\eqref{eq:A} is now $(0,0,r)^t$ and
$\rk A\le1$, we have $D=E=0$, or equivalently
\begin{equation}\label{eq:two-relations}
 gw+hv=0,\qquad aw+bv=0.
\end{equation}
The $X$- and $T$-rows of $H$ are supported in columns $Y,Z$ and form
\[
 \begin{pmatrix}gv&hw\\av&bw\end{pmatrix}.
\]
Its determinant vanishes, since
\[
 vw(gb+ha)=vb(gw+hv)+vh(aw+bv)=0.
\]
There is only one further nonzero row of $H$, again contradicting
$\rk H=3$. This proves $B=0$.

We have proved $\ell=f=i=r=u=0$.
The induced operator on the one-dimensional space $S/W$ is nilpotent,
hence zero. Thus $D=0$, which is the first relation
in~\eqref{eq:cross}. The first row of $H$ is now
\[
 (0,gv+js,hw+\kappa\tau,0,0).
\]
It cannot vanish, since only two other rows can be nonzero. This proves
the second assertion of~\eqref{eq:cross} and completes the proof.
\end{proof}

\section{The quartic cover and its cubic resolvent}\label{sec:resolvent}

Let $C/k$ satisfy the hypotheses of Lemma~\ref{lem:normal}.
The pencil $|\theta|$ defines a morphism $\rho:C\to\PP^1$ of degree
four. On the affine chart $Z=1$, write $(X,Y,Z)=(x,x^2,1)$.
Its equations take the form
\begin{equation}\label{eq:additive-cover}
 \begin{pmatrix}T^2\\U^2\end{pmatrix}
 +A(x)\begin{pmatrix}T\\U\end{pmatrix}
 =\begin{pmatrix}P_2(x)\\P_3(x)\end{pmatrix},\qquad
 A(x)=\begin{pmatrix}
 gx^2+h&jx^2+\kappa\\
 sx^2+\tau&vx^2+w
 \end{pmatrix},
\end{equation}
where $P_2,P_3$ have degree at most four.
The two equations are monic with relatively prime leading monomials
$T^2,U^2$ for a degree order in $T,U$ over $k[x]$. Their quotient
algebra is therefore free of rank four over $k[x]$, with basis
$1,T,U,TU$.

\begin{lemma}\label{lem:branch}
The morphism $\rho$ is separable and has a unique geometric branch
value. The fiber over that value is either $4P$ or $2P_1+2P_2$ with
$P_1\ne P_2$.
\end{lemma}
\begin{proof}
The relative Jacobian determinant of~\eqref{eq:additive-cover} is
\begin{align*}
 \det A(x)
 &=(gv+js)x^4+(gw+hv+j\tau+\kappa s)x^2+hw+\kappa\tau\\
 &=\alpha x^4+\beta,
\end{align*}
where $(\alpha,\beta)\ne(0,0)$ by~\eqref{eq:cross}.
In homogeneous coordinates $(S_0:S_1)$ on the base, the determinant is
\begin{equation}\label{eq:branch-form}
 \alpha S_0^4+\beta S_1^4.
\end{equation}
On the second chart, put $t=x^{-1}$, $T'=t^2T$ and $U'=t^2U$.
After multiplication by $t^4$, the linear coefficient matrix is
$t^2A(t^{-1})$, with determinant $\alpha+\beta t^4$.
Thus~\eqref{eq:branch-form} describes the determinant on both charts.
The cover is finite flat, and the relative Jacobian criterion
identifies the non-\'etale locus with the vanishing of this
determinant. Over $k$, it is the fourth power of a nonzero linear
form. Therefore $\rho$ is separable and has exactly one geometric
branch value; this is infinity when $\alpha=0$.

For a fixed geometric base point, the fiber is a torsor under the
kernel of the additive map
\[
 (T,U)\longmapsto (T^2,U^2)+A(T,U).
\]
This kernel is a finite group scheme of length four. Translation by
a geometric point of the fiber identifies the fiber with the kernel.
At the branch value the kernel is nonreduced. Its geometric points
form an $\F_2$-vector space, and translations identify their local
rings. Hence it has one or two geometric points, with equal local
lengths. Since $C$ is smooth, the fiber divisor is respectively
$4P$ or $2P_1+2P_2$.
\end{proof}

\subsection{The cubic resolvent}

Over the unramified locus, a degree-four cover determines an \'etale
cover of degree three by the action of $S_4$ on the three partitions
of a four-element set into two unordered pairs. Its smooth projective
completion is called the \emph{cubic resolvent}.
This construction commutes with base change.
For a torsor under $\F_2^2$, the three partitions correspond to the
three nonzero translation vectors: a vector $v$ gives the pairs
$\{w,w+v\}$. This correspondence is Galois-equivariant and independent
of a chosen origin in the torsor. Thus the generic cubic resolvent
of~\eqref{eq:additive-cover} is the nonzero part of the homogeneous
additive kernel.

\begin{lemma}\label{lem:resolvent}
Suppose that the exceptional fiber is $2P_1+2P_2$, with $P_1\ne P_2$.
The cubic resolvent is geometrically irreducible and has genus zero.
Above the branch value, its map to $\PP^1$ has two geometric points,
with ramification indices two and one.
\end{lemma}
\begin{proof}
Move the branch value to infinity by changing the basis of
$H^0(\theta)$. The span of $Y,Z$ is preserved, so the equations
retain the form~\eqref{eq:normal}. Moreover, the determinant remains
a nonzero fourth power under this change; renormalizing $T,U$ only
multiplies it by a nonzero scalar. On the chart at infinity, the linear
coefficient matrix specializes to
\[
 A_\infty=\begin{pmatrix}g&j\\s&v\end{pmatrix}.
\]
Its determinant is zero. The matrix itself is nonzero, since
$A_\infty=0$ would give the kernel $T^2=U^2=0$ and a fiber with only
one geometric point. Hence $\rk A_\infty=1$.

Let $e$ be the unique nonzero geometric point of this kernel. It
satisfies $A_\infty e=e^{(2)}$. Choose a nonzero vector $f_0$ in
$\ker A_\infty$. The vectors $e,f_0$ are independent, since
$A_\infty e=e^{(2)}\ne0$.
If $P$ is the matrix with columns $e,f_0$, then
\[
 (P^{(2)})^{-1}A_\infty P=\diag(1,0).
\]
After this change of $T,U$ and renormalization of their squares, we
have $g=1$ and $j=s=v=0$.
Equation~\eqref{eq:cross} gives $w=0$, while nonvanishing of the
determinant on the affine base gives $\kappa\tau\ne0$.
The homogeneous kernel is therefore
\begin{equation}\label{eq:kernel}
 T^2+(x^2+h)T+\kappa U=0,\qquad U^2+\tau T=0.
\end{equation}
Every nonzero kernel point has $U\ne0$. Eliminating $T$ yields
\begin{equation}\label{eq:resolvent-cubic}
 U^3+\tau(x^2+h)U+\kappa\tau^2=0.
\end{equation}
Choose $a_0^2=\tau$ and $b_0^2=h$, and set $z=x+U/a_0+b_0$.
Then
\[
 z^2=\frac{\kappa\tau}{U},
 \qquad
 U=\frac{\kappa\tau}{z^2},
 \qquad
 x=z+\frac{\kappa a_0}{z^2}+b_0.
\]
To prove irreducibility, regard~\eqref{eq:resolvent-cubic} as a
quadratic equation for $x$ over $k(U)$:
\[
 x^2=\frac{U^2}{\tau}+h+\frac{\kappa\tau}{U}.
\]
The right-hand side has a simple pole at $U=0$, so it is not a square.
The defining polynomial is primitive in $k[U][x]$, hence irreducible
in $k[x,U]$ and in $k(x)[U]$. The inverse rational expressions above
identify its function field with $k(z)$.
Thus the resolvent is geometrically irreducible of genus zero.
The rational function $x=z+\kappa a_0/z^2+b_0$ has degree three and poles
of orders two and one at $z=0$ and $z=\infty$, respectively.
These are the two points above the branch value, with the asserted
ramification indices.
\end{proof}

\subsection{Descent and rational points}

\begin{proof}[Proof of Theorem~\ref{thm:structural}]
Let $C/\F_Q$ satisfy the hypotheses of the theorem.
Proposition~\ref{prop:trigonal} excludes the trigonal case.
The canonical theta characteristic and its two-dimensional section
space are defined over $\F_Q$, so the degree-four morphism $\rho$ is
defined over $\F_Q$. By
Lemma~\ref{lem:branch}, this morphism has a unique geometric branch
value $b$. Its branch locus is defined over $\F_Q$, so $b$ is rational.

If $\rho^{-1}(b)=4P$, the unique geometric point $P$ is rational and
$\theta\sim4P$. Hence $K_C\sim8P$, as asserted in~\textup{(i)}.

Otherwise the fiber has two distinct geometric points and
Lemma~\ref{lem:resolvent} applies. The cubic resolvent is intrinsic, so
its smooth projective completion $R$ is defined over $\F_Q$.
It is geometrically irreducible of genus zero. Over a finite field
such a curve is isomorphic to $\PP^1$, and hence
\[
 \#R(\F_Q)=Q+1.
\]
The two geometric points of $R$ above $b$ have different ramification
indices. Frobenius fixes each of them, so both are rational.
Consequently exactly $Q-1$ points of $R(\F_Q)$ lie above the unramified
part of the base.

For $a\in\PP^1(\F_Q)\setminus\{b\}$, write
\[
 n_a=\#\rho^{-1}(a)(\F_Q),\qquad r_a=\#R_a(\F_Q).
\]
These are the numbers of fixed points of Frobenius on a four-element
set and on its three pair-partitions. The possibilities are
\begin{equation}\label{eq:permutations}
\begin{array}{c|ccccc}
 \text{cycle type}&1^4&2\,1^2&2^2&3\,1&4\\\hline
 n_a&4&2&0&1&0\\
 r_a&3&1&3&0&1
\end{array}
\end{equation}
and in every case $n_a\le1+r_a$.
There are $Q$ rational unramified base values. Summing gives
\[
 \sum_{a\ne b}n_a\le Q+\sum_{a\ne b}r_a=Q+(Q-1)=2Q-1.
\]
The exceptional fiber contributes at most two rational points.
Thus $\#C(\F_Q)\le2Q+1$, proving~\textup{(ii)}.
\end{proof}

\section{Maximal curves over \texorpdfstring{$\F_{64}$}{F64}}\label{sec:application}

\begin{proof}[Proof of Theorem~\ref{thm:main}]
Suppose that $C/\F_{64}$ is maximal of genus five. Then
\[
 \#C(\F_{64})=64+1+2\cdot5\cdot8=145.
\]
The curve is not hyperelliptic. Indeed, the unique hyperelliptic
involution and the quotient map descend to $\F_{64}$. The quotient is
a genus-zero curve over a finite field, so a degree-two map gives
\[
 \#C(\F_{64})\le2(64+1)=130.
\]

By Lemma~\ref{lem:maximal}, $\Ca^3=0$.
Theorem~\ref{thm:structural} therefore applies. Its first alternative
gives a rational point $P$ with $K_C\sim8P$, contrary to
Lemma~\ref{lem:maximal}. Its second alternative gives
\[
 145=\#C(\F_{64})\le2\cdot64+1=129,
\]
again a contradiction. Hence no such maximal curve exists.
Finally,~\eqref{eq:known-values} gives $N_{64}(5)=140$.
\end{proof}

The theorem also completes the genus spectrum over this field.
Write
\[
 \mathcal M(Q)=\{g\ge0:\text{there is a maximal curve of genus $g$ over $\F_Q$}\}
\]
when $Q$ is a square.

\begin{corollary}\label{cor:spectrum}
The genus spectrum of maximal curves over $\F_{64}$ is
\[
 \mathcal M(64)=\{0,1,2,3,4,6,7,9,10,12,28\}.
\]
\end{corollary}
\begin{proof}
The inclusions and exclusions in~\cite[Propositions~4.1(2), 4.2(2)
and Remark~4.4(1)]{ATT} leave only the possible additional genus five. Theorem~\ref{thm:main} excludes it.
\end{proof}

\begin{remark}
The hypothesis $\Ca^3=0$ is essential in Lemma~\ref{lem:normal}.
It makes Frobenius square-zero on its three-dimensional image.
This property does not follow from $2$-rank zero or supersingularity
alone. For example, the smooth supersingular curve
\[
 X^2+YZ=T^2+ZT+XY=U^2+UZ+XT+YT=0
\]
is the specialization $b_1=b_2=0$, $b_3=1$ of
\cite[Theorem~1.1]{Dra26}. Formula~\eqref{eq:HW-formula} gives
$F(e_Z)=e_X$, $F(e_X)=e_T$ and $F(e_T)=e_U$.
Thus $F^3\ne0$, and consequently $\Ca^3\ne0$.
\end{remark}

\Needspace{18\baselineskip}
\appendix
\section{The coefficient matrix}\label{app:HW}

For the quadrics $q_1=X^2+YZ$ and~\eqref{eq:general}, put
\begin{align*}
 A_0&=b\ell+fw+hu+i\tau+\kappa r,\\
 B_0&=a\ell+fv+gu+is+jr,\\
 D&=d\ell+fu+gw+hv+ir+j\tau+\kappa s,\\
 E&=aw+bv+du+ip+jn+\kappa m,\\
 G&=a\tau+bs+dr+fp+gn+hm.
\end{align*}
In the order $X,Y,Z,T,U$, formula~\eqref{eq:HW-formula} gives
\begin{equation}\label{eq:full-HW}
\begingroup
\setlength{\arraycolsep}{4pt}
\renewcommand{\arraystretch}{1.3}
\begin{pmatrix}
 D&c\ell+gv+js&e\ell+hw+\kappa\tau&\ell&0\\
 A_0&B_0&0&0&0\\
 B_0&0&A_0&0&0\\
 E&av+cu+io+jm&bw+eu+i\zeta+\kappa n&f\ell+u&i\\
 G&as+cr+fo+gm&b\tau+er+f\zeta+hn&r&f+i\ell
\end{pmatrix}.
\endgroup
\end{equation}
For example, the $XX$-entry is the coefficient of $X^2YZTU$ and the
$XT$-entry is the coefficient of $YZT^3U$ in $q_1q_2q_3$.
The former is $D$, while the latter is $\ell$.
The vanishing of the $Y$- and $Z$-rows is therefore precisely
\eqref{eq:zero-row1}--\eqref{eq:zero-row2}.

\section{The numerical bounds}\label{app:numerical}

We verify the two numerical inputs recorded in~\cite{many}.
The exclusion below uses~\cite[Theorem~3.1]{HL}; the construction
attaining $140$ points is due to Howe.

\begin{proposition}\label{prop:defects}
There is no curve of genus five over $\F_{64}$ with
$141$, $142$, $143$ or $144$ rational points.
\end{proposition}
\begin{proof}
Suppose that $\#C(\F_{64})=145-d$, where $1\le d\le4$.
Write its real Weil polynomial as $h(x)=\prod_{i=1}^5(x-x_i)$, so that
the Weil polynomial is $t^5h(t+64/t)$. Then $-16\le x_i\le16$ and
$\sum_i(x_i+16)=d$. On removing the zero roots from $h(y-16)$, we
obtain a monic polynomial $f(y)\in\mathbb Z[y]$ with $r$ positive
roots of sum $d$. Their product is a positive integer, so the
arithmetic--geometric mean inequality gives $r\le d$. If $r=d$,
all roots are one.

For $r=1$, we have $f=y-d$. For $r=2$, we have
$f=y^2-dy+c$ with $1\le c\le\lfloor d^2/4\rfloor$.
The only remaining case is $r=3$, $d=4$. Write
$f=y^3-4y^2+by-c$. The elementary symmetric inequalities give
$1\le b\le5$ and $1\le c\le2$. Its discriminant is
\[
 16b^2-4b^3-256c-27c^2+72bc.
\]
Among these ten pairs it is nonnegative only for $(b,c)=(4,1),(5,2)$.
Thus the complete list is as follows; the last column indicates the
exclusion used below.
\[
\begin{array}{c|l|c}
 d&f(y)&\text{exclusion}\\\hline
 1&y-1&\mathrm{I}\\
 2&y-2&\mathrm{II}\\
  &(y-1)^2&\mathrm{I}\\
 3&y-3&\mathrm{I}\\
  &y^2-3y+1&\mathrm{I}\\
  &(y-1)(y-2)&\mathrm{II}\\
  &(y-1)^3&\mathrm{I}\\
 4&y-4&\mathrm{II}\\
  &y^2-4y+1&\mathrm{I}\\
  &y^2-4y+2&\mathrm{III}\\
  &(y-1)(y-3)&\mathrm{I}\\
  &(y-2)^2&\mathrm{II}\\
  &(y-1)(y^2-3y+1)&\mathrm{I}\\
  &(y-1)^2(y-2)&\mathrm{II}\\
  &(y-1)^4&\mathrm{I}
\end{array}
\]
In every case $h(x)=(x+16)^{5-r}h_0(x)$, where $h_0(x)=f(x+16)$.

\emph{I.} Here $h_0(0)$ is odd, so the corresponding abelian variety
is ordinary, and $h_0(-16)=f(0)\in\{\pm1,\pm3\}$ is squarefree.
Since $5-r>0$, \cite[Theorem~3.1]{HL}, with the square root of $64$
chosen to be $-8$, excludes a Jacobian with real Weil polynomial $h$.

\emph{II.} These candidates contain $x+14$ with multiplicity one or
two, or $x+12$ with multiplicity one. The associated Weil factors
are $t^2+14t+64$ and $t^2+12t+64$, whose normalized Newton slopes
are respectively $(1/6,5/6)$ and $(1/3,2/3)$. In the Newton polygon
of an abelian variety, the multiplicity of a slope has to be
divisible by its denominator. None of the indicated multiplicities
satisfies this condition; the remaining factors have different slopes.

\emph{III.} Here $h_0=x^2+28x+194$, giving the Weil factor
\[
 t^4+28t^3+322t^2+1792t+4096.
\]
Its normalized Newton slopes are $1/12,1/12,11/12,11/12$.
These also violate the denominator condition. The factor $(x+16)^3$
only adds slope $1/2$ and cannot remove the obstruction.
\end{proof}

\begin{proposition}\label{prop:lower}
There is a curve of genus five over $\F_{64}$ with $140$ rational points.
\end{proposition}
\begin{proof}
We use Howe's construction in~\cite{many}. Let $s\in\F_8$ satisfy
$s^3+s+1=0$, and put
\[
 q=sx^2+sy^2+s^3z^2+s^3xy+s^3yz+s^3xz.
\]
Let $D$ be the plane quartic $q^2+xyz(x+y+z)=0$, and let $C$ be the
smooth projective curve defined by the Artin--Schreier extension
\begin{equation}\label{eq:Howe}
 w^2+w=\frac{sx+s^2y}{z}
\end{equation}
of $\F_{64}(D)$. The partial derivatives of the quartic are
$yz(y+z)$, $xz(x+z)$ and $xy(x+y)$. Their common projective zeros
are the seven points of $\PP^2(\F_2)$, at which the quartic takes
nonzero values in $\{s^2,s^6,1\}$. Thus $D$ is smooth of genus three
and is nonhyperelliptic.

The only possible ramification of~\eqref{eq:Howe} is over $z=0$.
On putting $u=x/y$, $v=z/y$, the two points there are $(r,0)$ with
$r^2+s^2r+1=0$. They are defined over $\F_{64}$ and conjugate over
$\F_8$. The function $t=q(u,1,v)$ is a local parameter, and
$t^2=uv(u+1+v)$. Expansion at either point gives
\[
 \frac{su+s^2}{v}
 =\frac{(r+s)^2}{t^2}+\frac{r+s}{t}+s(r+1)+O(t).
\]
Indeed, one may substitute
$u=r+s^{-3}t+(s^6+r^{-1})t^2+O(t^3)$ in the two equations.
Replacing $w$ by $w+(r+s)/t$ removes both pole terms, proving that
the cover is unramified.

The extension is geometrically nontrivial. Otherwise its right-hand
side would equal $a^2+a$ over $\overline{\F}_{64}(D)$, where $a$ has
only simple poles at these two points. This would give a function
of degree at most two on the nonhyperelliptic curve $D$.
Riemann--Hurwitz for the resulting geometrically connected unramified
double cover gives $2g(C)-2=2(2\cdot3-2)=8$.

For clarity, the finite count needed for this example is
\[
\begin{array}{c|cc}
\operatorname{Tr}_{64/2}(sx+s^2y)&0&1\\\hline
\#\{(x,y)\in\F_{64}^2:q(x,y,1)^2+xy(x+y+1)=0\}&70&38.
\end{array}
\]
It is obtained by enumerating the $64^2$ pairs in
$\F_2[\alpha]/(\alpha^6+\alpha+1)$, taking
$s=\alpha^3+\alpha^2+\alpha$, and evaluating
$\operatorname{Tr}_{64/2}(a)=a+a^2+a^4+a^8+a^{16}+a^{32}$.
At either point at infinity the reduced residue is $s(r+1)$, with
\[
 \operatorname{Tr}_{64/8}(s(r+1))=s^3,
 \qquad \operatorname{Tr}_{8/2}(s^3)=1.
\]
Hence these fibers give no rational points, and $\#C(\F_{64})=2\cdot70=140$.
\end{proof}

\section*{Acknowledgments}
S. Tafazolian was partially supported by CNPq grant no.~302774/2025-4,
FAEPEX grant no.~3485/25, and FAPESP grant no.~2024/00923-6.

\section*{Declaration of AI assistance}
AI-based tools were used solely for English-language editing.

\end{document}